\documentclass[11pt,twoside,reqno]{amsart}
\usepackage{amsmath,amsthm,amsfonts,amssymb}
\usepackage{latexsym}
\usepackage{mathrsfs}
\usepackage[all]{xy}
\usepackage{url}
\usepackage[pdftex]{graphicx}
\usepackage{float}
\usepackage{tikz}
\usepackage{longtable}
\usetikzlibrary{decorations.markings} 
\usepackage{hyperref}
\usepackage{ upgreek }

\allowdisplaybreaks

\makeatletter
\newtheorem*{rep@theorem}{\rep@title}
\newcommand{\newreptheorem}[2]{\newtheorem*{rep@#1}{\rep@title}
\newenvironment{rep#1}[1]{\def\rep@title{#2 \ref*{##1}}
\begin{rep@#1}}
{\end{rep@#1}}} 
\makeatother

\theoremstyle{plain}
  \newtheorem{thm}{Theorem}[section]
  \newreptheorem{thm}{Theorem}
  \newtheorem*{thm*}{Theorem}
  \newtheorem{prop}[thm]{Proposition}
  \newreptheorem{prop}{Proposition}
  \newtheorem{lem}[thm]{Lemma}
  \newtheorem{cor}[thm]{Corollary}
  
\theoremstyle{definition}
  \newtheorem{definition}[thm]{Definition}
  \newtheorem{example}[thm]{Example}
  
\theoremstyle{remark}
  \newtheorem{remark}[thm]{Remark}
\numberwithin{equation}{section}

\newcommand{\Z}{\mathbb{Z}}

\newcommand{\BS}{\mathrm{BS}}

\DeclareMathOperator{\lk}{lk}

\DeclareRobustCommand{\TT}{{\scalebox{1.5}[1.6]{$\tau$}}}
\DeclareRobustCommand{\TTStrict}{{\mathcal{T}_{<}}}
\DeclareRobustCommand{\TTMetric}{{\TT'}}
\DeclareRobustCommand{\TTMetricStrict}{{\TT'_{\hspace{-4pt}<}}}

\DeclareRobustCommand{\CC}{{\scalebox{1.6}[1.0]{$\upzeta$}}}
\DeclareRobustCommand{\CCStrict}{{\CC_{<}}}
\DeclareRobustCommand{\CCMetric}{{\CC'}}
\DeclareRobustCommand{\CCMetricStrict}{{\CC'_{<}}}

\def\HMF{\mathrm{C_{hm}}}
\def\HMV{\mathrm{T_{hm}}}
\def\HHarmonic{\mathcal{H}}
\newcommand{\cornerVertex}[1]{v({#1})}

\newcommand{\cornerFace}[1]{f({#1})} 

\newcommand{\cornerFirstWord}[1]{w_1({#1})} 
\newcommand{\cornerSecondWord}[1]{w_2({#1})}
\newcommand{\cornerRelation}[1]{r({#1})} 

\newcommand{\cornerFirstEdgeLength}[1]{\ell_1({#1})} 
\newcommand{\cornerSecondEdgeLength}[1]{\ell_2({#1})}
\newcommand{\cornerRelationLength}[1]{\ell_r({#1})} 
\newcommand{\cornerVertexLength}[1]{\ell_v({#1})}
\newcommand{\VV}{\mathcal{V}} 
\newcommand{\EE}{\mathcal{E}}
\newcommand{\FF}{\mathcal{F}} 
\newcommand{\vertices}[1]{\operatorname{\mathrm{vertices}}(#1)} 
\newcommand{\faces}[1]{\operatorname{\mathrm{faces}}(#1)} 

\newcommand{\graphFromPresentation}[1]{\Gamma({#1})}
\newcommand{\DontShow}[1]{} 
\def\bcr{\mathrel{\sim}} 
\def\tcbcr{\mathrel{\overline{\sim}}} 

\newcommand{\statementTheoremDR}{If a presentation $P$ satisfies condition $\TTMetric$ and has no proper powers, then it is DR.}

\newcommand{\statementTheoremArtin}{An Artin group $A_\Gamma$ is $2$-dimensional if and only if its standard presentation $P_\Gamma$ satisfies condition $\TTMetric$.}

\newcommand{\statementTheoremHyperbolic}{Let $G$ be a group which admits a finite presentation satisfying conditions $\TTMetricStrict$ and $C(3)$. Then $G$ is hyperbolic.}

\newcommand{\statementTheoremQuadraticDehn}{Let $P$ be a presentation satisfying conditions $\TTMetric-C'(\frac{1}{2})$ and such that all its relators have length $r$, then $P$ has a quadratic Dehn function. Moreover, if $P$ is finite the group $G$ presented by $P$ has solvable conjugacy problem.}

\newcommand{\statementPropTrivialWords}{Let $P=\langle X \mid R\rangle$ be a presentation of a group $G$, satisfying conditions $\TTMetric-C'(\frac{1}{2})$. Let $r_{\min}$ be the length of the shortest relator. Then any nontrivial word $W$ in the free group generated by $X$ representing the trivial element in $G$ has length at least $r_{\min}$. In particular, if $P$ has a relator of length greater than or equal to $2$, then $G$ is nontrivial.}
\begin{document}

\title[Generalized small cancellation conditions]{Generalized  small cancellation conditions, non-positive curvature and diagrammatic reducibility}

\author[M.A. Blufstein]{Mart\'in Axel Blufstein}
\author[E.G. Minian]{El\'ias Gabriel Minian}
\author[I. Sadofschi Costa]{Iv\'an Sadofschi Costa}
\address{Departamento  de Matem\'atica - IMAS\\
 FCEyN, Universidad de Buenos Aires. Buenos Aires, Argentina.}
\email{mblufstein@dm.uba.ar}
\email{gminian@dm.uba.ar}
\email{isadofschi@dm.uba.ar}



\thanks{Researchers of CONICET. Partially supported by grants PIP 11220170100357, PICT 2017-2997, PICT-2017-2806 and UBACYT 20020160100081BA}

\begin{abstract}
We present a metric condition $\TTMetric$  which describes the geometry of classical small cancellation groups and applies also to other known classes of groups such as two-dimensional Artin groups. We prove that presentations satisfying condition $\TTMetric$ are diagrammatically reducible in the sense of Sieradski and Gersten. In particular we deduce that the standard presentation of an Artin group is aspherical if and only if it is diagrammatically reducible. We show that, under some extra hypotheses, $\TTMetric$-groups have quadratic Dehn functions and solvable conjugacy problem. In the spirit of Greendlinger's lemma, we prove that if a presentation $P=\langle X \mid R\rangle$ of group $G$ satisfies conditions $\TTMetric-C'(\frac{1}{2})$, the length of any nontrivial word in the free group generated by $X$ representing the trivial element in $G$ is at least that of the shortest relator. We also introduce a strict metric condition $\TTMetricStrict$, which implies hyperbolicity. Finally we investigate non-metric and dual variants of these conditions, and study a harmonic mean version of small cancellation theory.
\end{abstract}

\maketitle

\section{Introduction}

The first ideas behind small cancellation theory appeared  more than one hundred years ago in the work of M. Dehn.
In \cite{Dehn} Dehn formulated the word and conjugacy problems and later he presented an algorithm that solved these problems for the fundamental groups of closed orientable surfaces of genus $g\geq 2$ \cite{Dehn2}.
The key property of the standard presentations of such groups is that any nontrivial product of two different cyclic permutations of the single relator or its inverse admits only a little cancellation (see \cite{LyndonSchupp}).
There are currently many variants of small cancellation theory.
The development of the classical theory started with the works of Tartakovski\u{\i} \cite{Tartakovskii}, Schiek \cite{Schiek}, Greendlinger \cite{Greendlinger1,Greendlinger2}, Lyndon \cite{Lyndon}, Schupp \cite{Schupp} (see also \cite{LyndonSchupp}), and continued later with the works of Rips \cite{Rips1,Rips2,Rips3} and Olshanski\u{\i} \cite{Olshanskii}.
The graphical variant of the theory started with Gromov \cite{Gromov03} and continued with the works of Ollivier \cite{Ollivier} and Gruber \cite{Gruber}.
The classical conditions of small cancellation require the pieces to be small with respect to the relators.
Concretely, in the metric condition $C'(\lambda)$ the length $\ell(s)$ of any piece $s$ of a relator $r$ is smaller than $\lambda \ell(r)$.
The non-metric variant $C(p)$ requires that no relator is a product of fewer than $p$ pieces.
In graphical small cancellation the presentations are constructed from a labeled graph and one requires the pieces in the graph to be small.
Small cancellation is not only related to the word and conjugacy problems.
It has applications to curvature problems, hyperbolicity, asphericity and diagrammatic reducibility.
The notion of diagrammatic reducibility first appeared in the work of Sieradski \cite{Sieradski} and was developed by Gersten \cite{Gersten}.
A presentation $P$ with no proper powers is diagrammatically reducible (DR) if all spherical diagrams over $P$ are reducible.
Diagrammatic reducibility, which is a stronger condition than asphericity, is relevant when studying equations over groups.
We refer the reader to Gersten's papers \cite{Gersten, GerstenBranchedCoverings} for more details on diagrammatic reducibility and applications to equations over groups.

In this article we introduce a metric condition $\TTMetric$ and a strict version $\TTMetricStrict$.
These conditions are defined in terms of sums of lengths of pieces incident to interior vertices of reduced diagrams (see Definition \ref{DefinitionTT}). They generalize the classical metric small cancellation conditions and apply also to other known classes of groups.
Our motivation is to encompass the geometries of classical small cancellation groups, $2$-dimensional Artin groups, and one-relator groups satisfying a certain condition $(T')$ recently introduced in \cite{BlufsteinMinian}, under the same unifying geometry. In the case of Artin groups, we show that condition $\TTMetric$ is equivalent to being two-dimensional.
\begin{repthm}{TheoremArtin}
\statementTheoremArtin
\end{repthm}
We show that presentations (and groups) satisfying conditions  $\TTMetric$ and $\TTMetricStrict$ have nice properties.
We analyze first curvature problems and diagrammatic reducibility. In this direction we obtain the following result.
\begin{repthm}{TheoremDR}
\statementTheoremDR
\end{repthm}
As an immediate consequence of the above theorem we deduce that the standard presentation of an Artin group is aspherical if and only if it is DR (see Remark \ref{RemarkArtinDR}).
We then prove that groups satisfying the strict condition $\TTMetricStrict$ are hyperbolic.
\begin{repthm}{TheoremHyperbolic}
\statementTheoremHyperbolic
\end{repthm}
We also prove that $\TTMetric$-groups have quadratic Dehn functions and solvable conjugacy problem, provided all the relators in the presentation have the same length.
\begin{repthm}{TheoremQuadraticDehn}
\statementTheoremQuadraticDehn
\end{repthm}
In fact we believe that the result is still valid without the hypothesis on the length of the relators although we could not  prove the general result. At the end of Section \ref{SectionQuadraticDehn} we discuss a strategy to prove solvability of the word problem for a wider class of groups using this result.

In particular Theorem \ref{TheoremQuadraticDehn} proves the existence of quadratic Dehn functions and the solution of the conjugacy problem for a subclass of $2$-dimensional Artin groups. The solution of the word problem for $2$-dimensional Artin groups was proved by Chermak \cite{Chermak}. Recently Huang and Osajda proved that all $2$-dimensional Artin groups have quadratic Dehn function and solvable conjugacy problem \cite{HuangOsajda}. Our results provide, although partially, an alternative, direct and simpler proof of these facts.

In the spirit of Greendlinger's lemma, we obtain the following result.
\begin{repprop}{propTrivialWords}
\statementPropTrivialWords
\end{repprop}

Conditions $\TTMetric$ and $\TTMetricStrict$ are defined in terms of the lengths of the pieces and relators incident to interior vertices of reduced diagrams over the presentations.
At first glance it may seem  that these conditions are difficult to check since the definitions require to analyze all interior vertices of every possible diagram over the presentation $P$.
However we prove below that, for finite presentations, these conditions can be verified by analyzing the directed cycles in a finite weighted graph $\Gamma(P)$ associated to $P$.
In this direction, in Section \ref{SectionAlgorithm} we describe an algorithm which decides whether a given finite presentation $P$ satisfies these conditions.
This algorithm has been implemented in the \textsf{GAP}\cite{GAP} package \textsf{SmallCancellation} \cite{SmallCancellation}.
As an immediate consequence of the construction of the algorithm, we also deduce that for any finite presentation $P$ satisfying the strict condition $\TTMetricStrict$, the curvature of the interior vertices of any diagram $\Delta$ over $P$ is bounded above by a negative constant $N$, which is independent of the diagram.
This fact is used in Theorem \ref{TheoremHyperbolic} to prove hyperbolicity of the groups satisfying condition $\TTMetricStrict$.

In Section \ref{SectionNonMetricAndDual} we introduce non-metric and dual analogues to conditions $\TTMetric$ and $\TTMetricStrict$.
Our non-metric conditions extend the classical non-metric small cancellation conditions.
We obtain generalizations of known results, similarly as for the metric conditions.
In particular we extend a result of Gersten on diagrammatic reducibility (see Theorem \ref{TheoremDRfornonmetric} below).
We also relate our non-metric conditions with previous results of Huck and Rosebrock \cite{HuckRosebrock}.
In contrast to the classical small cancellation conditions, there is no relation between our metric and non-metric conditions.
We illustrate this fact with some examples.
Finally, in Section \ref{SectionHarmonic} we introduce harmonic mean variants of the classical small cancellation conditions.

\section{The small cancellation condition $\TTMetric$}\label{sectionTTMetric}

Let $K$ be a combinatorial $2$-complex.
A diagram $\Delta$ in $K$ is a combinatorial map $\varphi \colon M\to K$ where $M$ is a combinatorial structure on the sphere to which, perhaps, we remove some open $2$-cells. 
This includes spherical diagrams (when $M$ is the whole sphere), (singular) disk diagrams (when $M$ is a sphere with one $2$-cell removed), and annular diagrams (the sphere with two $2$-cells removed).
As usual,  $0$-cells, $1$-cells and $2$-cells will be called respectively vertices, edges and faces. If $P=\langle X \ | \ R \rangle$ is a presentation of a group $G$, a diagram over $P$ is a diagram $\varphi \colon M\to K_P$ where $K_P$ is the standard $2$-complex associated to the presentation $P$.
Throughout this article we will assume that the relators of the presentations are cyclically reduced and no relator is a cyclic permutation of another relator or of the inverse of another relator.
Since $M$ is orientable we can fix an orientation in the usual way, so that when traversing the boundaries of the $2$-cells the edges in the intersection of two faces $f,f'$ are traversed twice, once in each possible orientation.
The map $\varphi \colon M\to K_P$ induces a labeling on the edges of $M$ by elements of $X$ and their inverses.
The label on the boundary of any oriented face of the diagram (starting at any vertex) is called a boundary label.
Note that boundary labels are elements in the set $R^*$ of all cyclic permutations of the elements of $R$ and their inverses.
A diagram $\Delta$ is \textit{reducible} if it contains two faces $f,f'$ such that the intersection of their boundaries $\partial f\cap \partial f'$ contains an edge such that the boundary labels of $f$ and $f'$ read with opposite orientations and starting at a vertex of this edge coincide, otherwise $\Delta$ is called \textit{reduced} (see \cite[Chapter V]{LyndonSchupp} for more details).
The degree $d(v)$ of a vertex $v$ in a diagram $\Delta$ is the number of edges incident to $v$ (the edges with both boundary vertices at $v$ are counting twice).
A vertex $v$ is called interior if $v\notin \partial M$. 

Given a reduced diagram $\Delta$ over a presentation $P$, we can remove all interior vertices of degree $2$.
We obtain a new diagram where the interior edges correspond to \textit{pieces} in $R^*$.
Recall that a piece is a word which is a common prefix of two different elements of $R^*$.
The length $\ell(e)$ of an interior edge $e$ in this new diagram is defined as the length of the corresponding word (equivalently, it is the number of edges of the original diagram that were glued together to obtain $e$).
In what follows we consider diagrams with no interior vertices of degree $2$.

Recall that the link of a vertex $v$ is an epsilon sphere about $v$ and the corners of the $2$-cells at $v$ correspond to edges in the link.
The endpoints of a corner in $M$ (of a $2$-cell $f$) at $v$ correspond to edges in the diagram incident to $v$ (see \cite{Gersten, WiseSectionalCurvature}). 
Given a corner $c$ at an interior vertex $v$, we denote by $\cornerFirstEdgeLength{c}$ and $\cornerSecondEdgeLength{c}$ the lengths of the incident edges and by  $\cornerRelationLength{c}$ the length of the relator $r\in R$ corresponding to the $2$-cell $f$.
Let $d'_F(v)=\sum_{c \ni v} \frac{\cornerFirstEdgeLength{c}+\cornerSecondEdgeLength{c}}{\cornerRelationLength{c}}$, where the sum is taken over all corners at $v$.

\begin{definition}\label{DefinitionTT}
 We say that a presentation $P$ satisfies the small cancellation condition $\TTMetric$ if for every interior vertex of any reduced diagram over $P$ (with no interior vertices of degree $2$), $d'_F(v) \leq d(v)-2$. Similarly, $P$ satisfies the strict small cancellation condition $\TTMetricStrict$ if for every interior vertex of any reduced diagram over $P$, $d'_F(v) < d(v)-2$. A group $G$ which admits a presentation $P$ satisfying condition $\TTMetric$ (resp. $\TTMetricStrict$) is called a $\TTMetric$-group (resp. $\TTMetricStrict$-group).
\end{definition}

We investigate now the first examples of presentations satisfying conditions $\TTMetric$ and $\TTMetricStrict$.

\subsection*{Classical metric small cancellation conditions} It is easy to verify that the classical metric small cancellation conditions $C'(1/6)$, $C'(1/4)-T(4)$ and $C'(1/3)-T(6)$ imply condition $\TTMetricStrict$. We will show below that finitely presented $\TTMetricStrict-C(3)$-groups are hyperbolic, generalizing the classical result for small cancellation groups (see \cite{BridsonHaefliger, Gromov}).

\subsection*{Two-dimensional Artin groups}

Let $\Gamma$ be a finite simple graph with a labeling on the edges by integers $m\geq 2$.
The Artin group defined by $\Gamma$ is the group $A_\Gamma$ given by the following presentation $P_\Gamma$.
The generators of $P_\Gamma$ are the vertices of $\Gamma$, and there is a relation of the form
$$\underbrace{ababa\cdots}_{m \text{ letters}}=\underbrace{babab\cdots}_{m\text{ letters}}$$
for every pair of vertices $a$ and $b$  connected by an edge labeled by $m$.
By results of Charney and Davis \cite{CharneyDavis, CharneyDavisFiniteKP1}, it is well-known that an Artin group $A_\Gamma$ is $2$-dimensional (i.e. it has geometric dimension $2$) if and only if for every triangle in the graph $\Gamma$ with edges labeled by $p$, $q$ and $r$ we have $\frac{1}{p} + \frac{1}{q} + \frac{1}{r} \leq 1$ (see also \cite{HuangOsajda}).

We will show that an Artin group is $2$-dimensional if and only if its standard presentation $P_\Gamma$ satisfies condition $\TTMetric$.
We will also prove below that any group which admits a finite presentation $P$ satisfying conditions $\TTMetric$ and $C'(\frac{1}{2})$ and with all relators of the same length, has quadratic Dehn function and solvable conjugacy problem (see Theorem \ref{TheoremQuadraticDehn}). These results put together partially recover, with an alternative and simpler proof, similar results for $2$-dimensional Artin groups recently obtained by Huang and Osajda \cite{HuangOsajda}. 

\begin{thm}\label{TheoremArtin}
\statementTheoremArtin
\end{thm}

\begin{proof}
Let $A_\Gamma$ be an Artin group and let $K$ be the $2$-complex associated to its standard presentation.
Note that the $2$-cells of $K$ have two distinguished sides in which all edges have the same orientation.
If the label of the edge in $\Gamma$ corresponding to the relator is $m$, each of these sides has $m$ edges.
The terminal vertices of both sides are called initial and final vertices of the relator, according to the orientation of the edges (cf. \cite[Section 4.1]{HuangOsajda}).

Let $\varphi \colon M\to K$ be a reduced diagram.
We analyze first the interior vertices of degree $3$.
It is easy to see that vertices of degree $3$ only correspond to intersections of faces in the diagram which are mapped to three different relators that form a triangle in the graph $\Gamma$.
Since they are three different relators, the length of the three pieces involved is $1$.
If the labels of the edges in the triangle are $p$, $q$ and $r$, then the equation for the condition $\TTMetric$ is the following:
$$\frac{1+1}{2p}+\frac{1+1}{2q}+\frac{1+1}{2r} \leq 3-2.$$

That is, $\frac{1}{p} + \frac{1}{q} + \frac{1}{r} \leq 1$, which is exactly the necessary and sufficient condition for the Artin group to be $2$-dimensional.

Now we prove that condition $\TTMetric$ is always satisfied in interior vertices of degree greater than or equal to $4$ (for any Artin group, not necessarily two-dimensional).
Note that such an interior vertex $v$ can be a terminal vertex or it can be inside of one of the sides of the $2$-cells containing it.
It is not difficult to see that if not all of the $2$-cells incident to $v$ are mapped to the same relator, at least four of them will have a piece of length $1$. Also, the longest piece in a $2$-cell with boundary of length $2n$ is $n-1$, and therefore the summands in condition $\TTMetric$ can be at most $\frac{2n-2}{2n}$.
In conclusion, if there are $2$-cells incident to the vertex which are mapped to different relators, there are at least four summands in the equation for condition $\TTMetric$ which are less than or equal to $\frac{1}{2}$.
Then condition $\TTMetric$ is satisfied.

We analyze now the case that the vertex has degree greater than or equal to $4$ and all the $2$-cells are mapped to the same relator.
Observe that if a $2$-cell contains $v$ in one of its sides, the summand corresponding to that $2$-cell is at most $\frac{1}{2}$.
Therefore there are at most three of such $2$-cells, and the rest have to contain $v$ as a terminal vertex.
If a $2$-cell $f$ contains $v$ as a terminal vertex, then the two adjacent $2$-cells to $f$ contain the vertex on a side.
This implies that we can reduce ourselves to the cases where $v$ has degree $4$, $5$ or $6$.

We look at the orientation of the edges incident to $v$, traversing them in clockwise order.
If we pass by a $2$-cell that has $v$ as a terminal vertex, the orientation of these edges is preserved, and if not, it is reversed.
Therefore the number of $2$-cells having $v$ on one of its sides is even.
Therefore, when $v$ has degree $5$ or $6$, there are at least four $2$-cells that have $v$ on a side.

It only remains to check the case where $v$ has degree $4$, two of the $2$-cells contain it as a terminal vertex and the other two on a side.
This situation is illustrated in Figure \ref{figone}.
Let $l_1,l_2,l_3,l_4$ be the lengths of the pieces involved and let $2n$ be the length of the relator.

\begin{figure}[h]
\begin{tikzpicture}
\filldraw (0,0) circle (2pt)
(4,0) circle (2pt)
(8,0) circle (2pt)
(2,1) circle (2pt)
(6,1) circle (2pt)
(2,-1) circle (2pt)
(6,-1) circle (2pt);
\draw (0,0) .. controls (2,1.35) .. (4,0);
\draw (4,0) .. controls (6,1.35) .. (8,0);
\draw (0,0) .. controls (2,-1.35) .. (4,0);
\draw (4,0) .. controls (6,-1.35) .. (8,0);
\draw (2,1) .. controls (4,1.5) .. (6,1);
\draw (2,-1) .. controls (4,-1.5) .. (6,-1);
\draw (3.1,0.84) node {$l_1$};
\draw (4.9,0.86) node {$l_2$};
\draw (3.1,-0.85) node {$l_3$};
\draw (4.9,-0.8) node {$l_4$};
\draw (4.3,0) node {$v$};
\end{tikzpicture}
\caption{A vertex with degree $4$ with the $2$-cells mapping to the same relator.}
\label{figone}
\end{figure}
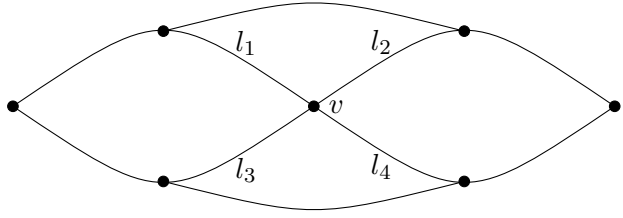

Here condition $\TTMetric$ can be rewritten as:
$$\frac{l_1+l_2}{2n}+\frac{l_2+l_4}{2n}+\frac{l_3+l_4}{2n}+\frac{l_1+l_3}{2n} = \frac{2l_1+2l_2+2l_3+2l_4}{2n}\leq 2.$$
Since $l_1+l_2\leq n$ and $l_3+l_4\leq n$, condition $\TTMetric$ is satisfied.
\end{proof}

\begin{remark}\label{RemarkArtinDR}
Combined with Theorem \ref{TheoremDR}, Theorem \ref{TheoremArtin} implies that the standard presentation $P_\Gamma$ of an Artin group is aspherical (i.e. $A_\Gamma$ is $2$-dimensional) if and only if $P_\Gamma$ is diagrammatically reducible (DR, for short).
Recall that a presentation $P$ with no proper powers is DR if all spherical diagrams over $P$ are reducible.
Note that being DR is in general a stronger condition than being aspherical (see \cite{Gersten, GerstenBranchedCoverings}).
\end{remark}

\subsection*{One-relator groups} 
In \cite{BlufsteinMinian} the first two authors introduced a small cancellation condition $(T')$ to study hyperbolicity of one-relator groups.
Condition $\TTMetricStrict$ generalizes condition $(T')$ to any presentation and Theorem \ref{TheoremHyperbolic} below provides an alternative and simpler proof of \cite[Thm. 3.1]{BlufsteinMinian} for one-relator groups.

\begin{example}
The following one-relator presentation does not satisfy conditions $C(6)$ nor 
$T(4)$, but it is $\TTMetric$. It also does not fall under the hypothesis of 
\cite[Thm. 3.1]{BlufsteinMinian} since it is $C'(\frac{1}{2})$, but not 
$C'(\frac{1}{4})$.
$$\langle a,b \mid a^3b^4a^3b^4(b^4a^3b^4a^3)^{-1} \rangle$$

\end{example}

\subsection*{Cyclic presentations}
The claims in the following examples can be verified using the \textsf{GAP} \cite{GAP} package \textsf{SmallCancellation} \cite{SmallCancellation}.

\begin{example}
The following cyclic presentation of a superperfect group satisfies condition $\TTMetricStrict$ but is not $C(6)$ nor $T(4)$.

$$\langle x_0, x_1, x_2,x_3,x_4 \mid x_{i+4}^{-1}x_{i+1}^{-1}x_i^{-1}(x_{i+4}x_{i+1})^2  \text{ for } i=0,\ldots, 4\rangle $$

\end{example}

\begin{example}
The following cyclic presentation of a superperfect group satisfies condition $\TTMetric-C'(\frac{1}{2})$
$$\langle x_0, \ldots, x_6 \mid x_{i+1} x_{i}^{-1} x_{i+6} x_{i+1}^{-1} x_{i} x_{i+6}^{-1} x_{i+2}^{-1} \text{ for } i=0,\ldots, 6\rangle$$
Then by Theorem \ref{TheoremQuadraticDehn} this group has a quadratic Dehn function.
However the group is not $\TTMetricStrict$, nor $C(6)$, $C(4)-T(4)$ or $C(3)-T(6)$.
\end{example}

\section{Non-positive curvature, diagrammatic reducibility and hyperbolicity}\label{sectionNonPositiveCurvature}

We recall first some basic notions on combinatorial curvature. Given a combinatorial $2$-complex $K$, we can assign a real number $w(c)$ to the corners, which we think of as angles.
This assignment is a {\em weight function} for the complex.
A finite combinatorial $2$-complex together with such a weight function is called an angled complex (see \cite{Gersten,WiseNonpositiveImmersions,WiseSectionalCurvature}).

Let $K$ be an angled complex. If $v$ is a vertex of $K$, its curvature is defined as $$\kappa(v) = 2\pi - \pi \chi(\lk_v) - \sum_{c \ni v}w(c).$$
Here $\chi(\lk_v)$ denotes the Euler characteristic of the link of $v$, and the sum is taken over all corners at $v$.
The curvature of a face $f$ is defined as $$\kappa(f) =  2\pi-\pi \ell(\partial f) + \sum_{c\in f}w(c),$$ where the sum is taken over all the corners in $f$ and $\ell(\partial f)$ is the number of edges in the boundary of $f$.
The following result can be found in \cite{BallmannBuyalo,WiseSectionalCurvature}.

\begin{thm}[Combinatorial Gauss--Bonnet Theorem]
Let $K$ be an angled $2$-complex. Then
$$\sum_{f\in \faces{K}}\kappa(f) +\hspace{-10pt} \sum_{v\in \vertices{K}}\hspace{-5pt} \kappa(v) = 2\pi\chi(K).$$
\end{thm}

\subsection*{Assignment of a weight function} Let $P$ be a presentation satisfying condition $\TTMetric$ or $\TTMetricStrict$. Given a reduced diagram $f\colon M\to K_P$ over $P$, we define the following weight function in $M$.
The weight of a corner $c$ at an interior vertex $v$ is $w(c)=\pi-\frac{\cornerFirstEdgeLength{c} + \cornerSecondEdgeLength{c}}{\cornerRelationLength{c}}\pi$ (recall that there are no interior vertices of degree $2$).
The weight of a corner $c$ at a vertex $v\in\partial M$ of degree $2$ is $w(c)=\pi$. If $c$ is a corner at a vertex $v\in\partial M$ of degree greater than $2$, we define $\cornerFirstEdgeLength{c}$ and $\cornerSecondEdgeLength{c}$ similarly as we did with interior vertices (the lengths of the incident edges obtained if we remove the vertices of degree $2$) and  $w(c)=\pi-\frac{\cornerFirstEdgeLength{c} + \cornerSecondEdgeLength{c}}{\cornerRelationLength{c}}\pi$.

With this assignment, the curvature of the faces of $M$ is $0$, and the curvature of the interior vertices is non-positive if $P$ satisfies condition $\TTMetric$, and strictly negative if $P$ satisfies condition $\TTMetricStrict$.

We will show that presentations satisfying condition $\TTMetric$ and without proper powers are DR and that finitely presented $\TTMetricStrict-C(3)$-groups are hyperbolic. Given a $2$-complex $M$, we denote by $\VV(M),\EE(M)$ and $\FF(M)$ the number of vertices, edges and faces of $M$ respectively.

\begin{thm}\label{TheoremDR}
\statementTheoremDR
\end{thm}
\begin{proof}
Since $P$ has no proper powers, our notion of reduced spherical diagram over $P$ coincides with that of \cite{Gersten}.
Therefore, in order to prove that $P$ is DR, we only have to verify that there are no reduced spherical diagrams over $P$.
Suppose $\varphi \colon M\to K_P$ is a reduced spherical diagram. We have the following identities:
\begin{align*}
\EE(M) &= \frac{1}{2}\sum_{v\in\vertices{M}}d(v),\\
\FF(M) &= \frac{1}{2}\sum_{v\in \vertices{M}}d'_F(v).
\end{align*}
The first one is clear, and the second one comes from the fact that in the right hand side of the second equality we are summing two times the length of each relator, divided by the length of each relator.
That is, we are summing $2$ for each face. Then
$$2 = \VV(M)-\EE(M)+\FF(M) = \VV(M)-\frac{1}{2}\sum_{v\in \vertices{M}}d(v) + \frac{1}{2}\sum_{v\in \vertices{M}}d'_F(v) \leq 0,$$
where the last inequality holds because all the vertices in the sphere are interior vertices.
This is a contradiction, and therefore $P$ is diagrammatically reducible.
\end{proof}

\begin{thm}\label{TheoremHyperbolic}
\statementTheoremHyperbolic
\end{thm}

\begin{proof}
We show that a finite presentation $P$ satisfying conditions $\TTMetricStrict$ and $C(3)$ has a linear isoperimetric inequality.

Let $\varphi \colon M\to K_P$ be a reduced disk diagram. We assign weights to the corners in $M$ as we did before. Then, by the Combinatorial Gauss--Bonnet Theorem,
$$2\pi \hspace{4pt} = \hspace{-8pt}\sum_{v\in \vertices{M}} \kappa(v) + \sum_{f\in\faces{M}}\ \kappa(f) \hspace{4pt}=\hspace{-8pt} \sum_{v\in\vertices{M}} \kappa(v).$$

Since $P$ satisfies $\TTMetricStrict$, then $\kappa(v)<0$ for every interior vertex $v$, and since $P$ is finite, by Corollary \ref{coroN} there is a constant $N<0$, which is independent of the diagram, such that $\kappa(v) \leq N$ for every interior vertex $v$. Then
\begin{align*}
2\pi &\leq \VV^{\circ}(M)N + \sum_{v\in\vertices{\partial M}}\kappa (v) \\
&\leq \VV^{\circ}(M)N + \VV(\partial M)\pi \\
&= \VV(M) N + \ell(\partial M)(\pi-N),
\end{align*}
where $\VV^{\circ}(M)$ denotes the number of interior vertices of $M$, and $\ell(\partial M)$ is the length of the boundary.
Then,
$$-\VV(M) N \leq \ell(\partial M)(\pi-N)-2\pi$$
and therefore
$$\VV(M) \leq \ell(\partial M)\frac{\pi-N}{-N}+\frac{2\pi}{N}.$$

Now, since $P$ satisfies condition $C(3)$, the number of faces can be linearly bounded by the number of vertices in the diagram.
Consequently, the number of faces in the diagram is linearly bounded by the length of its boundary.
\end{proof}

\section{Quadratic Dehn function and conjugacy problem}\label{SectionQuadraticDehn}

In this section we will show that a finitely presented group which admits a presentation $P$ satisfying conditions $\TTMetric$ and $C'(\frac{1}{2})$ and such that all relators of $P$ have the same length $r$, has a quadratic Dehn function and solvable conjugacy problem.

Let $\varphi \colon M\to K_P$ be a diagram over $P$.
The {\em boundary layer} $L$ of $M$ consists of every vertex in the boundary of $M$, every edge incident to a vertex in the boundary, and every open face with a vertex in the boundary.
Note that $L$ is usually not a combinatorial complex.
Let $M_1=M\backslash L$ be the complement of the boundary layer.
Note that $M_1$ is a subcomplex of $M$.
The following lemma will be used to prove the main result of this section.

\begin{lem}\label{Lemmaboundarylayer}
Let $P$ be a presentation satisfying conditions $\TTMetric-C'(\frac{1}{2})$ and such that all its relators have length $r$, and let $\varphi \colon M\to K_P$ be an annular or disk diagram over $P$. Then
$$\VV(\partial M_1) \leq \VV(\partial M)-r\chi(M).$$
\end{lem}
\begin{proof}
We had previously removed interior vertices of degree $2$ from the diagrams.
We subdivide the boundary of $M_1$ reintroducing the vertices of degree $2$.
We still denote this diagram by $M$.

For the vertices of $M$ of degree greater than $2$, we assign weights to the corners as before, and for vertices of degree $2$, both weights are equal to $\pi$. With this assignment, every face has curvature $0$ and all interior vertices have non-positive curvature, since $P$ satisfies $\TTMetric$.

In what follows, we can assume without loss of generality that $M$ is non-singular, since we are going to bound the length of the boundary of $M_1$ in terms of the length of the boundary of $M$.
Since $P$ is $C'(\frac{1}{2})$ we may assume that each boundary $2$-cell $f$ has at least two edges which are not on the boundary of $M$, for otherwise we can remove $f$ decreasing the length of the boundary and without changing $M_1$.
In particular this reduction allows us to assume that $M_1\neq \emptyset$.

Let $B$ be the following complex. We take the disjoint union of the $0$-cells, $1$-cells and $2$-cells (now closed) of the boundary layer of $M$ and we identify the boundaries of the closed $2$-cells only in the vertices and edges of the boundary layer of $M$.
We omit vertices of degree $2$ in the cell structure of $B$.
Note that $\ell(\partial B) = \ell (\partial M ) +\ell(\partial M_1)$.
If $M$ is a disk, $B$ is a planar and connected combinatorial complex, so its Euler characteristic is less than or equal to $1=\chi(M)$.
If $M$ is an annulus, $B$ may have more than one connected component, but none of them would be a disk, since they all have a disconnected complement.
Therefore its Euler characteristic is less than or equal to $0=\chi(M)$.

We separate its vertices into two sets: $V_1$ will denote the set of vertices of $B$ that are in the boundary of $M$, and $V_2$ the set of remaining vertices of $B$. 
Since $P$ satisfies condition $\TTMetric$, by Gauss--Bonnet we have
$$2\pi\chi(M) \leq \sum_{v\in V_1}\kappa(v).$$
Also by Gauss--Bonnet, we have
$$\sum_{v\in \vertices{B}}\kappa(v) = 2\pi\chi(B).$$
Therefore
$$\sum_{v\in V_2}\kappa(v) \leq 0.$$
Now since each boundary $2$-cell has at least two edges which are not on the boundary of $M$ we have
$$\VV_1+ \sum_{v\in V_2}\sum_{c\ni v}1  =\VV(B)+\EE(B)-\ell(\partial B)  = \VV_2 +\sum_{v\in V_1}\sum_{c\ni v} 1 - \VV^\circ(B).$$
Now putting everything together
\begin{align*}
2\pi\chi(M) &\leq \sum_{v\in V_1}\kappa(v) - \sum_{v\in V_2}\kappa(v)\\
&= \sum_{v\in V_1}\left(\pi-\sum_{c \ni v}\left(\pi-\frac{\cornerFirstEdgeLength{c}+\cornerSecondEdgeLength{c}}{r}\pi\right)\right) - \sum_{v\in V_2}\left(\pi-\sum_{c \ni v}\left(\pi-\frac{\cornerFirstEdgeLength{c}+\cornerSecondEdgeLength{c}}{r}\pi\right)\right)\\
& =  \sum_{v\in V_1}\sum_{c \ni v}\frac{\cornerFirstEdgeLength{c}+\cornerSecondEdgeLength{c}}{r}\pi - \sum_{v\in V_2}\sum_{c \ni v}\frac{\cornerFirstEdgeLength{c}+\cornerSecondEdgeLength{c}}{r}\pi \\
&= \frac{2\VV(\partial M)-2\VV(\partial M_1)}{r}\pi.
\end{align*}
It follows that
$$\VV(\partial M_1) \leq \VV(\partial M)-r\chi(M).$$
\end{proof}

\begin{thm}\label{TheoremQuadraticDehn}
\statementTheoremQuadraticDehn
\end{thm}

\begin{proof}[Proof of the quadratic Dehn function]
Let $\varphi \colon M\to K_P$ be a reduced disk diagram. We show that $\VV(M) \leq \frac{1}{r}\VV(\partial M)^2$ by induction in the number of interior vertices of $M$. We have
\begin{align*}
\VV(M) &= \VV(M_1)+\VV(\partial M)\\ 
& \leq \frac{1}{r}\VV(\partial M_1)^2+\VV(\partial M) \\
&\leq \frac{1}{r}(\VV(\partial M)-r)^2+\VV(\partial M) \\
&= \frac{1}{r}\VV(\partial M)^2 - 2\VV(\partial M) + r + \VV(\partial M) \\
&\leq \frac{1}{r}\VV(\partial M)^2.
\end{align*}
The first inequality follows by induction, and the second one follows from Lemma \ref{Lemmaboundarylayer}.
Finally, since $P$ satisfies $C'(\frac{1}{2})$, each face of $M$ has at least three sides. Then we can bound the number of faces of $M$ by the number of vertices of $M$, obtaining the desired quadratic isoperimetric inequality.
\end{proof}

It follows that the finitely generated groups which admit presentations satisfying the hypotheses of Theorem \ref{TheoremQuadraticDehn} have solvable word problem. This will be used to prove that they also have solvable conjugacy problem. 

We attack the conjugacy problem following the strategy of \cite[Section V.7]{LyndonSchupp}).

\begin{remark}\label{RemarkConjugacy}
Let $P = \langle X\ |\ R \rangle$ be a finite presentation of a group $G$ with solvable word problem. Suppose that all the relators have the same length $r$.
Let $w_1$ and $w_2$ be words in the free group $F(X)$.
We write $w_1 \bcr w_2$ if there exists a word $b$ in $F(X)$ with $|b|<r$ such that $bw_1b^{-1}w_2^{-1}=1$ in $G$.
Here $|b|$ denotes the length of the word $b$. Since the word problem is solvable, the relation $\bcr$ is decidable.
Now let $u$ and $v$ be cyclically reduced words in $F(X)$ and let $d=|u|+|v|$. Take $W = \{w\in F(X), |w|\leq d\}$.
Note that $W$ is finite since $X$ is finite. Note also that the set $W$ depends on the lengths of $u$ and $v$ and that $u,v\in W$.
We write $u \tcbcr v$ if there exist words $w_1,\ldots ,w_k$ in $W$ such that $u \bcr w_1 \bcr \ldots \bcr w_k \bcr v$.
Equivalently, $\tcbcr $ is the transitive closure in $W$ of the relation $\bcr$.
Note that this relation is also decidable since $W$ is finite.
In order to prove that the conjugacy problem is solvable it suffices to prove that if two words $u,v\in F(X)$ are conjugate in $G$, then $u\tcbcr v$.
\end{remark}

We will also use the following result of Schupp \cite[Section V.7]{LyndonSchupp}. Let $A$ be an annular diagram and $L$ its boundary layer.
The diagram $A_1=A\backslash L$ (the complement of $L$ in $A$) may be disconnected, but it has at most one annular component.
A simply connected component of $A_1$ is called a {\em gap}.
Let $K_1,\ldots,K_n$ be the gaps. Then $H=A\backslash(L\cup\bigcup_{i=1}^n K_i)$ is the annular component of $A_1$, assuming there is any.
Let $\sigma$ and $\tau$ be the outer and inner boundaries of $H$.
A pair $(D_1,D_2)$ of faces (not necessarily distinct) in $A$ is called a {\em boundary linking pair} if $\sigma \cap \partial D_1 \neq \emptyset$, $\partial D_1 \cap \partial D_2 \neq \emptyset$, and $\partial D_2 \cap \tau \neq \emptyset$.

\begin{lem}[Schupp]
Let $A$ be an annular diagram having at least one region, and let $H$ be the diagram obtained by removing its boundary layer and its gaps. If there are no boundary linking pairs, $H$ is an annular diagram.
\end{lem}

\begin{proof}[Proof of the conjugacy problem]
Take $u,v\in F(X)$ cyclically reduced and suppose that they are conjugate in $G$.
Let $d=|u|+|v|$. By Remark \ref{RemarkConjugacy}, we only have to prove that $u\tcbcr  v$.
Let $A$ be an annular diagram with $u$ and $v^{-1}$ as inner and outer boundaries.
Construct the diagrams $A=H_0, H_1,\ldots, H_k$, where $H_{i+1}$ is obtained from $H_i$ by removing its boundary layer and its gaps, and let $H_k$ be the first of such diagrams with a linking pair. 
By Lemma \ref{Lemmaboundarylayer}, $\ell(\partial H_{i+1})\leq \ell(\partial H_i)$ for each $0 \leq i \leq k-1$.
Therefore, $\ell(\partial H_i) \leq d$ for every  $i$ (and so the boundary labels of $\partial H_i$ are in the set $W$).

Let $\sigma_i$ and $\tau_i$ be the outer and inner boundaries of $H_i$ respectively.
Let $S_i$ be the subdiagram of $M$ consisting of $\sigma_i$, $\sigma_{i+1}$ and all the cells of $M$ between these two paths.
Define $T_i$ in the same manner with respect to $\tau_i$ and $\tau_{i+1}$.
It is clear that any boundary face of $S_i$ intersects both boundaries of $S_i$.
So there is a path $\gamma_i$ from $\sigma_i$ to $\sigma_{i+1}$ with a label of length less than or equal to $r$.
Let $s_i$ and $s_{i+1}^{-1}$ be the labels of $\sigma_i$ and $\sigma_{i+1}$ starting at a given vertex.
Then $s_i \bcr s_{i+1}$.
Analogously, we have $t_i \bcr t_{i+1}$ where $t_i^{-1}$ and $t_{i+1}$ are the labels of $T_i$.

The last annulus $H_k$ has a boundary linking pair $(D_1, D_2)$.
We have vertices $v_0 \in \sigma_k \cap \partial D_1$, $v_1 \in \partial D_1 \cap \partial D_2$, and $v_2 \in \partial D_2 \cap \tau_k$.
Therefore there are paths $\beta_1$ and $\beta_2$ from $v_0$ to $v_1$ and from $v_1$ to $v_2$ labeled by words $b_1$ and $b_2$ of length smaller than or equal to $\frac{r}{2}$.
Let $\beta = \beta_1\beta_2$, then its label is a word of length less than $r$.
Let $s$ be the word read in the outer boundary of $H_k$ starting at $v_0$, and $t^{-1}$ the word read in the inner boundary of $H_k$ starting at $v_2$.
We have that $sb_1b_2t^{-1}b_2^{-1}b_1^{-1}=1$ in $G$. Then $s \bcr t$. 

Since $s_0$ and $t_0$ are cyclic permutations of $u$ and $v$ respectively, and $s$ and $t$ are cyclic permutations of $s_k$ and $t_k$ respectively, we have
$$u \tcbcr  s_0 \bcr s_1 \bcr \ldots \bcr s_k \tcbcr  s \bcr t \tcbcr  t_k \bcr \ldots \bcr t_0 \tcbcr  v.$$
\end{proof}

A slight modification in the proof of Lemma \ref{Lemmaboundarylayer} allows one to obtain a lower bound on the length of a the words which represent the trivial element in the group $G$, even if the relators have different lengths.

\begin{prop}\label{propTrivialWords}
\statementPropTrivialWords
\begin{proof}
Let $M$ be a reduced disk diagram.
We follow the same steps as in the proof of Lemma \ref{Lemmaboundarylayer}
and we get that
$$2\pi\chi(M) \leq \sum_{v\in V_1}\sum_{c \ni v}\frac{\cornerFirstEdgeLength{c}+\cornerSecondEdgeLength{c}}{l_r(c)}\pi - \sum_{v\in V_2}\sum_{c \ni v}\frac{\cornerFirstEdgeLength{c}+\cornerSecondEdgeLength{c}}{l_r(c)}\pi.$$
In particular since the terms $\frac{\ell_i(c)}{\ell_r(c)}$ in the first sum which do not correspond to edges in the boundary cancel with terms in the second sum, we have $r_{\min}\chi(M) \leq V(\partial M)$.
Therefore, since $M$ is a disk, $r_{\min}\leq V(\partial M)$, which implies that words representing the trivial element have length at least $r_{\min}$. 

For the second statement, by removing all the relators of length $1$ along with the corresponding generators, we can assume that each relator of $P$ has length at least $2$. Note that condition $C'(\frac{1}{2})$ guarantees that each of these generators can appear in only one relator.
\end{proof}

\end{prop}

In Theorem \ref{TheoremQuadraticDehn} we proved the existence of quadratic Dehn functions and solvability of the conjugacy problem for presentations satisfying conditions $\TTMetric-C'(\frac{1}{2})$, provided the relators have the same length. We believe that the result is still valid without the assumption on the lengths of the relators. We discuss now a strategy to prove solvability of the word problem for a  wider class of groups. Given a presentation $P = \langle X \ | \ R\rangle$ of a group $G$, our aim is to obtain a new presentation $P'$ of a group $H$ such that $G$ embeds in $H$, and such that $P'$ satisfies conditions $\TTMetric-C'(\frac{1}{2})$ and the relators have the same length. By Theorem \ref{TheoremQuadraticDehn}, this  would imply that $H$, and therefore $G$, has solvable word problem. To do so, we will choose a positive integer $n_x$ for some $x\in X$ and replace every occurrence of $x$ in the relators by $x^{n_x}$. If the element $x$ in the group $G$ has infinite order, this corresponds to adding an $n_x$--th root, or equivalently, to taking the amalgamated product of $G$ with $\Z$ along the subgroup $n_x\Z$. If we make these replacements for a finite number of $x\in X$, we obtain a new presentation $P'$ of an overgroup $H$ of $G$. Of course, it is not always possible to choose the $n_x$ so that all the relators in the new presentation have the same length and even if this is possible, the presentation $P'$ obtained may not satisfy conditions $\TTMetric-C'(\frac{1}{2})$ (even if $P$ does). The following example illustrates this technique.

\begin{example}
Consider the following presentation
$$P = \langle a,b,c,s,t \mid tats^{-1}b^{-1}s^{-1}, tbts^{-1}c^{-2}s^{-1}, tc^2ts^{-1}a^{-1}s^{-1} \rangle.$$
Note that the relators do not have the same length. This presentation does not satisfy conditions $C(5)$, 
$T(4)$ nor $\TTMetric$. Now, it is easy to see that $a$ and $b$ have infinite order in the group $G$ presented by $P$, and by choosing $n_a=2$ and $n_b=2$, we obtain the 
following presentation
$$P' = \langle a,b,c,s,t \mid ta^2ts^{-1}b^{-2}s^{-1}, tb^2ts^{-1}c^{-2}s^{-1}, tc^2ts^{-1}a^{-2}s^{-1} \rangle.$$
Note that all its relators have the same length. One can verify that $P'$ satisfies conditions $\TTMetric-C'(\frac{1}{2})$ (although it does not satisfy conditions $C(5)$ nor $T(4)$). This implies that $G$ has solvable word problem.
\end{example}

\section{Verifying condition $\TTMetric$}\label{SectionAlgorithm}

In this section we give an algorithm to verify if a finite presentation  $P=\langle X\mid R\rangle$ satisfies condition $\TTMetric$.
Note that a priori it is not clear that such an algorithm exists, since the definition involves checking something for every possible diagram over $P$.
The algorithm described here has been implemented in the \textsf{GAP}\cite{GAP} package \textsf{SmallCancellation} \cite{SmallCancellation}. 

We describe a weighted directed graph $\graphFromPresentation{P}$.
The vertices of this graph are the tuples $(r,p,q)$ such that
\begin{itemize}
\item $r\in R^*$,
\item $p$ and $q$ are pieces, and
 \item we can write $r=qsp$ without cancellations.
\end{itemize}
There is an edge $(r,p,q)\to (r', p', q')$ if 
\begin{itemize}
 \item $p'=q^{-1}$, and
 \item $r'\neq r^{-1}$.
\end{itemize}
The weight of this edge is $1-\frac{|p|+|q|}{|r|}$ (by simplicity, we divide by $\pi$ the weights that we considered in Section \ref{sectionNonPositiveCurvature}). The weight of a cycle is the sum of the weights of its edges.

Given a diagram $\varphi \colon M\to K_P$, we can fix an orientation in $M$ as explained in Section \ref{sectionTTMetric}.
The corners in the diagram inherit the orientations of the corresponding faces.
Note that if $c$ is a corner at an interior vertex $v(c)$,  then $(\cornerRelation{c}, \cornerFirstWord{c}, \cornerSecondWord{c})$ is a vertex in $\graphFromPresentation{P}$.
Here $\cornerRelation{c}$ denotes the relator read in the boundary of the face, starting from the vertex $v(c)$ and following the orientation of the face, $\cornerFirstWord{c}$ and $\cornerSecondWord{c}$ are the subwords written in the edges of the oriented corner (the first edge being the one oriented towards $\cornerVertex{c}$).
This remark and the following proposition make clear why this graph is meaningful.

\begin{prop}\label{propWeightCycles}

(i) Let $v$ be an interior vertex in a reduced diagram $\varphi \colon M\to K_P$.
Then there is a directed cycle $\gamma$ in $\graphFromPresentation{P}$ of length at least $3$ and weight $d(v)-d'_F(v)$.

(ii) Let $\gamma$ be a directed cycle in $\graphFromPresentation{P}$ of length at least $3$ and weight $w$. Then there is a reduced diagram over $P$ and an interior vertex $v$ such that $d(v)-d'_F(v)=w$.
\begin{proof}
We first prove (i).
Let $v$ be an interior vertex in a reduced diagram $\varphi \colon M\to K_P$.
Let $c_1,\ldots, c_n$ be the corners around $v$, numbered clockwise.
Then $\cornerFirstWord{c_{i+1}}=\cornerSecondWord{c_i}^{-1}$ (indices are modulo $n$).
Since the diagram is reduced we have $r(c_{i+1})^{-1}\neq r(c_i)^{-1}$ and therefore there is an edge
$$   (\cornerRelation{c_i},\cornerFirstWord{c_i},\cornerSecondWord{c_i})\xrightarrow{e_i} (\cornerRelation{c_{i+1}},\cornerFirstWord{c_{i+1}},\cornerSecondWord{c_{i+1}})$$ in $\graphFromPresentation{P}$ with weight $1-\frac{\cornerFirstEdgeLength{c_i}+\cornerSecondEdgeLength{c_i}}{\cornerRelationLength{c_i}}$.
Then the cycle $\gamma = (e_1,\ldots, e_n)$ has weight $d(v)-d'_F(v)$.

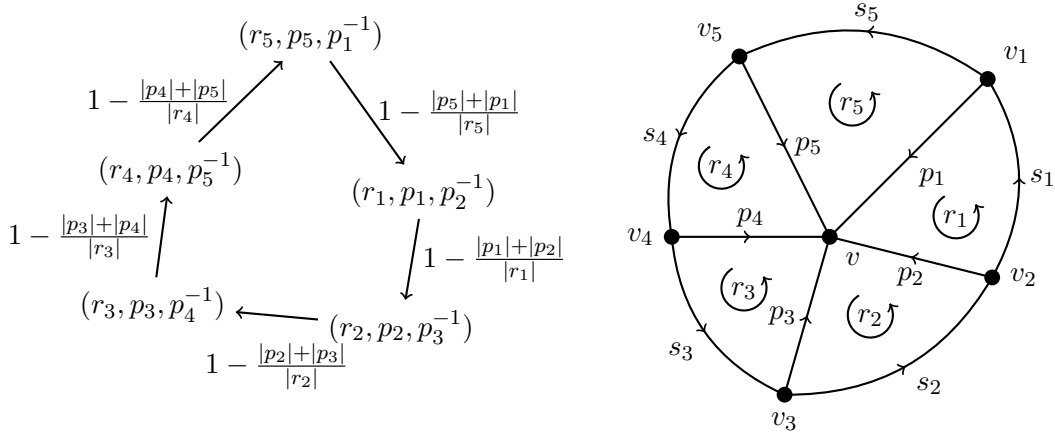
\begin{figure}[H]
\begin{tikzpicture}[dot/.style = {circle, fill, minimum size=#1, inner sep=0pt, outer sep=0pt},
dot/.default = 6pt, scale=0.6]

\node[dot, label=below right:{$v$}] (O) at (0, 0){};
\node[dot, label=above right:{$v_1$}] (X1) at (3.5, 3.5){};
\node[dot, label= right:{$v_2$}] (X2) at (3.6, -0.9){};
\node[dot, label=below:{$v_3$}] (X3) at (-1, -3.5){};
\node[dot, label=left:{$v_4$}] (X4) at (-3.5,0){};
\node[dot, label=above left:{$v_5$}] (X5) at (-2, 4){};
\node (R1) at (2.8,0.4) {$r_1$};
\node (R2) at (0.9,-1.8) {$r_2$};
\node (R3) at (-1.9,-1.2) {$r_3$};
\node (R4) at (-2.4,1.5) {$r_4$};
\node (R5) at (0.5,2.9) {$r_5$};
\begin{scope}[thick, decoration={
    markings,
    mark=at position 0.5 with {\arrow{>}}}
    ] 
\draw[postaction={decorate}] (X1.center) -- (O.center) node[midway, below right] {$p_1$};
\draw[postaction={decorate}] (X2.center) --  (O.center) node[midway,below ] {$p_2$};
\draw[postaction={decorate}] (X3.center) --  (O.center)node[midway,left ] {$p_3$};
\draw[postaction={decorate}] (X4.center) --  (O.center)node[midway,above ] {$p_4$};
\draw[postaction={decorate}] (X5.center) -- (O.center)node[midway,right ] {$p_5$};
\end{scope}

\begin{scope}[thick, decoration={
    markings,
    mark=at position 0.5 with {\arrow{>}}},
    ] 

\draw (X2.center) edge[bend right, postaction={decorate}]node[midway,right ]{$s_1$} (X1.center);
\draw (X3.center) edge[bend right, postaction={decorate}]node[midway,below right ]{$s_2$} (X2.center);
\draw (X4.center) edge[bend right, postaction={decorate}]node[midway,below left ]{$s_3$} (X3.center);
\draw (X5.center) edge[bend right, postaction={decorate}]node[midway,left ]{$s_4$} (X4.center);
\draw (X1.center) edge[bend right, postaction={decorate}]node[midway,above ]{$s_5$} (X5.center);
\end{scope}

\begin{scope}[thick,decoration={
    markings,
    mark=at position 1 with {\arrow{>}}}
    ] 

\node (T) at (-0.25,0.5){};
\draw[postaction={decorate}] (R1)+(T) arc (120:390:0.5);
\draw[postaction={decorate}] (R2)+(T) arc (120:390:0.5);
\draw[postaction={decorate}] (R3)+(T) arc (120:390:0.5);
\draw[postaction={decorate}] (R4)+(T) arc (120:390:0.5);
\draw[postaction={decorate}] (R5)+(T) arc (120:390:0.5);

\end{scope}
\node (V1) at (-9,1) {$(r_1,p_1,p_2^{-1})$};
\node (V2) at (-9.5,-2) {$(r_2,p_2,p_3^{-1})$};
\node (V3) at (-15,-1.5){$(r_3,p_3,p_4^{-1})$};
\node (V4) at (-14.6,1.5) {$(r_4,p_4,p_5^{-1})$};
\node (V5) at (-11.5,4.5) {$(r_5,p_5,p_1^{-1})$};
\begin{scope}[thick,decoration={
    markings,
    mark=at position 1 with {\arrow{>}}}
    ] 
\draw[->] (V1) -- (V2) node[midway,right ] {$1-\frac{|p_1|+|p_2|}{|r_1|}$};
\draw[->] (V2) -- (V3) node[midway,below, inner sep=10pt ] {$1-\frac{|p_2|+|p_3|}{|r_2|}$};
\draw[->] (V3) -- (V4) node[midway,left ] {$1-\frac{|p_3|+|p_4|}{|r_3|}$};
\draw[->] (V4) -- (V5) node[midway,left ] {$1-\frac{|p_4|+|p_5|}{|r_4|}$};
\draw[->] (V5) -- (V1) node[midway,right ]  {$1-\frac{|p_5|+|p_1|}{|r_5|}$};

\end{scope}
\end{tikzpicture}
\caption{On the left a cycle $\gamma$ in $\graphFromPresentation{P}$, on the right the corresponding diagram constructed in the proof of part (ii) of Proposition \ref{propWeightCycles}.}
\label{figuretwo}
\end{figure}

We now prove (ii).
Let $n\geq 3$ and let $\gamma$ be a cycle in $\graphFromPresentation{P}$ of length $n$.
By the first condition for the edges of  $\graphFromPresentation{P}$, the vertices of $\gamma$ can be named
$(r_1,p_{1},p_{2}^{-1})$, $(r_2,p_{2},p_{3}^{-1})$, $\ldots,$ $(r_n,p_{n},p_{1}^{-1})$.
For each $i$ we consider the word $s_i$ such that $r_i = p_{i+1}^{-1}s_ip_i$ without cancellations.
We construct a disk diagram $\Delta$ with $n+1$ vertices, $2n$ edges and $n$ faces as follows.
The vertices of $\Delta$ will be denoted by $v,v_1,\ldots, v_n$.
For each $i$ the diagram has an edge $v_i  \xrightarrow{e_i} v$ which reads $p_i$ and an edge $v_{i+1}\xrightarrow{\alpha_i} v_{i}$ which reads $s_i$.
For each $i$ there is a face $f_i$ attached with boundary $(e_{i+1}^{-1},\alpha_i,e_i)$ which reads $r_i$ (starting at $v$) (see Figure \ref{figuretwo}).
By the second condition for an edge in $\graphFromPresentation{P}$, the diagram is reduced.
Note that by construction, $d(v)-d'_F(v)$ is the weight of $\gamma$.
\end{proof}
\end{prop}

\begin{cor}\label{coroAlgorithmTTMetric}
A presentation $P$ satisfies condition $\TTMetric$ if and only if each directed cycle in $\graphFromPresentation{P}$ of length at least $3$ has weight greater than or equal to $2$.
 
A presentation $P$ satisfies condition $\TTMetricStrict$ if and only if each directed cycle in $\graphFromPresentation{P}$ of length at least $3$ has weight greater than $2$.
\end{cor}

Note that Corollary \ref{coroAlgorithmTTMetric} gives an algorithm to check if a finite presentation satisfies $\TTMetric$, for it is possible to use Dijkstra's algorithm to find the least weight of a directed cycle of length at least $k$ in a directed graph with positive edge weights.
This can be done by constructing an auxiliary graph having $(k+1)$ vertices for each vertex in the original graph.
For more details on this see the implementation in the \textsf{GAP} package \textsf{SmallCancellation}\cite{SmallCancellation}.

From Proposition \ref{propWeightCycles} we deduce the following result, which is used in the proof of Theorem \ref{TheoremHyperbolic}.

\begin{cor}\label{coroN}
If a finite presentation $P$ satisfies condition $\TTMetricStrict$ there is a constant $N<0$ such that $\kappa(v) \leq N$ for every diagram $\Delta$ and every interior vertex $v\in \Delta$.
\begin{proof}
Since the weights are positive, we can take $N$ to be $-\pi$ times the minimum weight of a simple directed cycle of length at least $3$ in $\graphFromPresentation{P}$. Note that, since the graph $\graphFromPresentation{P}$ is finite, there is a finite number of such cycles.
\end{proof}

\end{cor}

The following examples of groups which do not satisfy $\TTMetric$ are consistent with our conjecture that $\TTMetric-C'(\frac{1}{2})$ implies a quadratic isoperimetric inequality even if the presentation has relators of different lengths.

\begin{example}
The presentations A and B from \cite{WLYL} have unsolvable word problem.
A \textsf{GAP} computation using \textsf{SmallCancellation} shows that these presentations do not satisfy $\TTMetric$.
\end{example}

\begin{example}
From \cite{BMS} we know the Baumslag--Solitar group $\BS(p,q)$ has exponential Dehn function if $|p|\neq |q|$.
Therefore, by Theorem \ref{TheoremQuadraticDehn} the groups $\BS(n,n+1)$ do not satisfy $\TTMetric$.
It can be seen that the minimum of $d(v)-d'_F(v)$ for $v$ an interior vertex in a diagram for the usual presentation of $\BS(n,n+1)$ is $2-\frac{1}{2n+3}$, which tends to $2$ as $n\to \infty$.
Note that $\BS(n,n)$ satisfies $\TTMetric-C'(\frac{1}{2})$ so by Theorem \ref{TheoremQuadraticDehn} one can verify the well-known fact that these groups have quadratic Dehn function.
\end{example}

\begin{example}
In \cite[Lemma 11]{BMS} a family of groups $M_{c,d}$ is considered and it is proved that the Dehn function of $M_{c,d}$ has order $n^{c+d}$.
We have that $M_{1,1}$ (which is a RAAG) satisfies $\TTMetric-C'(\frac{1}{2})$.
Some \textsf{GAP} computations suggest that the minimum of $d(v)-d'_F(v)$ for these groups is $\frac{23}{20}$ for any $(c,d)\neq (1,1)$.
\end{example}

\begin{example}
In \cite{BMS} it is proved that the Dehn function of the group $E= \langle b,s,t\mid s^{-1}bs =b^2, t^{-1}bt=b\rangle$ is at least $2^n$.
This group does not satisfy $\TTMetric$ (the minimum of $d(v)-d'_F(v)$ is $\frac{8}{5}$).
\end{example}

\begin{example}
 In \cite[Chapter 8]{ECHLPT} it is shown that the Heisenberg group $\langle a,b,z\,\mid\, z=[a,b],\, [a,z],\, [b,z]\rangle$ is not automatic.
 In \cite{Edjvet} it is shown that the Heisenberg group admits a cyclic presentation
 $$P=\langle x_0,x_1\mid
 x_0^{-1}x_1x_0x_1^{-1}x_0^{-1}x_0x_1,\, 
 x_1^{-1}x_0x_1x_0^{-1}x_1^{-1}x_1x_0\rangle.$$
 For this presentation the minimum of $d(v)-d'_F(v)$ is $0$.
 This implies that the presentation is not DR, there is a reduced spherical diagram with two $0$-cells.
\end{example}
It would be interesting to know more about what the minimum of $d(v)-d'_F(v)$ says about a presentation.

\section{Non-metric and dual conditions}\label{SectionNonMetricAndDual}

In this section we comment the non-metric analogues to conditions $\TTMetric$ and $\TTMetricStrict$ and their dual versions. 

\subsection*{Dual conditions}
Similarly as in the case of classical small cancellation conditions $C$ and $T$, it is natural to consider the dual notions of conditions $\TTMetric$ and $\TTMetricStrict$.

Let $\varphi \colon M\to K_P$ be a diagram over a presentation $P$.
We denote by $\cornerVertexLength{c}$ the sum of the lengths of all the edges incident at the vertex of $c$.
Now we define $d'_V(f) = \sum_{c \in f} \frac{\cornerFirstEdgeLength{c}+\cornerSecondEdgeLength{c}}{\cornerVertexLength{c}}$, where the sum is over all the corners of $f$.

\begin{definition}
We say that a presentation $P$ satisfies the small cancellation condition $\CCMetric$ if for every interior face $f$ of any reduced diagram over $P$, $d'_V(f) \leq d(f)-2$.
Similarly, $P$ satisfies the strict small cancellation condition $\CCMetricStrict$ if for every interior face of any reduced diagram over $P$, $d'_V(f) < d(f)-2$.
A group $G$ which admits a presentation $P$ satisfying condition $\CCMetric$ (resp. $\CCMetricStrict$) is called a $\CCMetric$-group (resp. $\CCMetricStrict$-group).
\end{definition}

Once again, condition $\CCMetricStrict$ generalizes the classical metric small cancellation conditions.
With similar proofs to the ones exhibited for conditions $\TTMetric$ and $\TTMetricStrict$ one can obtain, for conditions $\CCMetric$ and $\CCMetricStrict$, results analogous to Theorem \ref{TheoremDR}, Theorem \ref{TheoremHyperbolic} and Theorem \ref{TheoremQuadraticDehn}.
However, neither of these conditions is implied by the other.
For example, Artin groups of dimension $2$ satisfies $\TTMetric$ but not $\CCMetric$.
The following is an example of the converse situation.

\begin{example}
Let $P=\langle a,b,c,x,y,z \ | \ abx, cby, ac^{-1}z\rangle$. It is easy to check that this presentation satisfies $\CCMetric$. However Figure \ref{figurethree} shows that it does not satisfy condition $\TTMetric$.
\end{example}

\begin{figure}[H]
\begin{tikzpicture}
\filldraw (0,0) circle (2pt)
(-2,-1) circle (2pt)
(2,-1) circle (2pt)
(0,2) circle (2pt);
\begin{scope}[very thick,decoration={
    markings,
    mark=at position 0.5 with {\arrow{>}}}
    ] 
    \draw[postaction={decorate}] (0,2)--(0,0);
    \draw[postaction={decorate}] (-2,-1)--(0,0);
    \draw[postaction={decorate}] (0,0)--(2,-1);
    \draw[postaction={decorate}] (2,-1)--(0,2);
    \draw[postaction={decorate}] (-2,-1)--(0,2);
    \draw[postaction={decorate}] (2,-1)--(-2,-1);
\end{scope}

\draw (0.2,1) node {$a$};
\draw (1,-0.3) node {$b$};
\draw (-1,-0.3) node {$c$};
\draw (1.3,0.6) node {$x$};
\draw (0,-1.3) node {$y$};
\draw (-1.3,0.6) node {$z$};
\end{tikzpicture}
\caption{A reduced diagram not satisfying condition $\TTMetric$.}
\label{figurethree}
\end{figure}
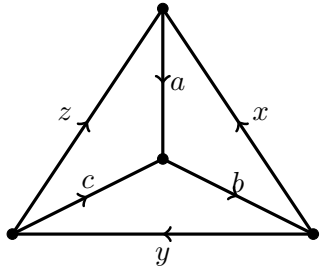

\subsection*{Non metric conditions}

Conditions $\TTMetric$, $\TTMetricStrict$, $\CCMetric$ and $\CCMetricStrict$ are metric conditions, in the sense that one measures the length of the relators and pieces. Using the same ideas one can generalize the classical non metric small cancellation conditions.

First we will define the non metric analogues to conditions $\CCMetric$ and $\CCMetricStrict$.
Let $P$ be a presentation, and $\varphi \colon M\to K_P$ a diagram over $P$.
Given a $2$-cell $f$ of $M$, we define $d_V(f) = \sum_{v\in f}\frac{2}{d(v)}$.
Note that this definition would coincide with $d'_V(f)$ if the length of every edge were equal to $1$.
Informally, we count the number of pieces instead of measuring their length.

\begin{definition}
We say that a presentation $P$ satisfies the small cancellation condition $\CC$ if for every interior face $f$ of any reduced diagram over $P$, $d_V(f) \leq d(f)-2$. Similarly, $P$ satisfies the strict small cancellation condition $\CCStrict$ if for every interior face of any reduced diagram over $P$, $d_V(f) < d(f)-2$. A group $G$ which admits a presentation $P$ satisfying condition $\CC$ (resp. $\CCStrict$) is called a $\CC$-group (resp. $\CCStrict$-group).
\end{definition}

One can similarly define $d_F(v)$ for a vertex in a diagram obtain non metric conditions corresponding to $\TTMetric$ and $\TTMetricStrict$. Given a reduced diagram over a presentation $P$, we define $d_F(v) = \sum_{c\ni v}\frac{2}{d(\cornerFace{c})}$, where we sum over all the corners at $v$ and $\cornerFace{c}$ is the face corresponding to $c$.

\begin{definition}
We say that a presentation $P$ satisfies the small cancellation condition $\TT$ if for every interior vertex of any reduced diagram over $P$, $d_F(v) \leq d(v)-2$, and every face has degree at least $3$. Similarly, $P$ satisfies the strict small cancellation condition $\TTStrict$ if for every interior face of any reduced diagram over $P$, $d_F(v) < d(v)-2$, and every face has degree at least $3$. A group $G$ which admits a presentation $P$ satisfying condition $\TT$ (resp. $\TTStrict$) is called a $\TT$-group (resp. $\TTStrict$-group).
\end{definition}

Non-metric conditions $\CC$, $\TTStrict$ and $\CCStrict$ appeared in \cite{HuckRosebrock} as $W$, $V^*$ and $V$ respectively. A condition similar to $\TT$ also appeared as condition $W^*$, although condition $W^*$ is defined for vertex reduced diagrams. 

\begin{remark}
Condition $\CC$ generalizes classical small cancellation conditions $C(6)-T(3)$, $C(4)-T(4)$ and $C(3)-T(6)$ (see \cite{HuckRosebrock}). 
\end{remark}

Similarly to the metric case one has the following generalization of a result by Gersten \cite[Remark 4.18]{Gersten}.

\begin{thm}\label{TheoremDRfornonmetric}
Let $P$ be a presentation satisfying condition $\CC$ or $\TT$  and without proper powers. Then it is DR.
\end{thm}

Unlike classical small cancellation, where the metric condition $C'(\frac{1}{\lambda})$ implies $C(\lambda+1)$, here there is no implication between the metric and non metric conditions, as the following example shows.

\begin{example}
Consider the following presentation $P =\langle x,y \ | \ y^2xy^{-2}x^{-1} \rangle$ from \cite[Example 2]{HuckRosebrock}.  In \cite{HuckRosebrock} it is shown that $P$ does not satisfy their condition $W^*$. In particular $P$ does not satisfy condition $\TT$. However, it is not difficult to verify that it satisfies condition $\TTMetric$.
\end{example}

\section{Harmonic small cancellation conditions}\label{SectionHarmonic}

Recall that a face $f$ in a diagram $M$ is said to be \textit{interior} if $f\cap \partial M =\emptyset$.

\begin{definition}
We say that a presentation $P$ satisfies the  \textit{harmonic small cancellation condition $\HMF(p)$} if, for every interior face $f$ in every reduced diagram the harmonic mean of the degrees of the neighbour faces of $f$ is at least $m$.
\end{definition}

\begin{definition}
We say that a presentation $P$ satisfies the  \textit{harmonic small cancellation condition $\HMV(q)$} if, for every interior face $f$ in every reduced diagram the harmonic mean of the degrees of the vertices of $f$ is at least $m$.
\end{definition}

\begin{remark}
Note that $C(p)\Rightarrow \HMF(p)$ and $T(q)\Rightarrow \HMV(q)$.
\end{remark}

The two propositions below follow from the definition of the harmonic small cancellation conditions by considering the sum over the faces of the diagram.
\begin{prop}
 If $P$ satisfies $\HMF(p)$ then every reduced spherical diagram over $P$ has $\FF\leq \frac{2}{p} \EE$.
\end{prop}

\begin{prop}
 If $P$ satisfies $\HMV(q)$ then every reduced spherical diagram over $P$ has $\VV\leq \frac{2}{q} \EE$.
\end{prop}

The following results generalize the classical results on $C(p)-T(q)$.

\begin{cor}
 If $P$ satisfies $\HMF(p)-\HMV(q)$ and $\frac{1}{p}+\frac{1}{q}\leq \frac{1}{2}$ then $P$ is diagrammatically reducible. 
\begin{proof}
Suppose there is a reduced spherical diagram over $P$. By the previous propositions we have
$$ 2= \chi(S^2)= \VV-\EE+\FF \leq \frac{2}{q} \EE - \EE + \frac{2}{p} \EE \leq 0,$$
a contradiction.
\end{proof}
\end{cor}

\begin{example}
 If $n\geq 5$, the presentation 
 $$\langle a_1, b_1,\ldots, a_n,b_n \mid [a_1,b_1][a_2,b_2]\cdots[a_n,b_n], a_1a_2a_3a_4a_5 \rangle$$ is not $C(6)$ but
 satisfies $\HMF(6)$ so it is DR.
\end{example}

\begin{remark}
 Since conditions $\HMF(p)$ and $\HMV(q)$ make sense for rational values of $p$ and $q$,
 so the previous result can be applied for example to presentations that do not satisfy neither $C(6)$ nor $T(4)$ but which satisfy $\HMF(4+\frac{1}{3})-\HMV(3+\frac{1}{2})$.
\end{remark}

\begin{definition}[A combined harmonic small cancellation condition]
Say that a presentation $P$ satisfies condition $\HHarmonic$ if for every reduced diagram and every interior face $f$, the harmonic mean of the degrees of all vertices and neighbor faces of $f$ is at least $4$.
\end{definition}
From the definition it is easy to see that $\HMF(p)-\HMV(q)$ implies $\HHarmonic$ provided that $\frac{1}{p}+\frac{1}{q}\geq \frac{1}{2}$.
Moreover, with the same argument used above we see that $\HHarmonic$ implies DR.

\begin{remark}
The harmonic small cancellation conditions $\HMF(p)$, $\HMV(q)$ and $\HHarmonic$ can be decided by an algorithm.
\end{remark}

\begin{remark}
 It is possible to prove, with some care and some additional mild hypotheses that the group presented by a $\HMF(p)-T(q)$ presentation is hyperbolic if $\frac{1}{p}+\frac{1}{q}<\frac{1}{2}$.
\end{remark}

\begin{remark}
If we include the central face $f$ in the average we obtain similar definitions satisfying similar properties that we shall not discuss here.
We can also consider dual conditions (considering the harmonic mean around a vertex instead of around a face).
Replacing the harmonic mean by the arithmetic mean seems not to be enough to obtain similar results.
\end{remark}

\bibliographystyle{plain}
\bibliography{references}

\end{document}